\documentclass[12pt,reqno]{amsart}
\usepackage{amssymb}
\usepackage{amsmath}
\usepackage{latexsym}
\usepackage{amsthm}
\usepackage{epsfig}
\usepackage{color}
\usepackage{comment}
\usepackage{amsfonts,amssymb,mathrsfs,amscd}

\usepackage{enumitem}

\theoremstyle{plain}
\newtheorem{theorem}{Theorem}[section]
\newtheorem{proposition}{Proposition}[section]
\newtheorem{corollary}{Corollary}[section]

\newtheorem{remark}{\bf Remark}[section]
\theoremstyle{definition}
\newtheorem{definition}{Definition}[section]

\renewcommand{\div}{\mathop\mathrm{div}}
\newcommand{\Tr}{\mathop\mathrm{Tr}}
\def\R{\mathbb R}

\def\({\left(}
\def\){\right)}

\def\Tr{\operatorname{Tr}}

\begin{document}
 \title[Lieb--Thirring and other
bounds for orthonormal systems]{Applications of the Lieb--Thirring
and other bounds for orthonormal systems in mathematical
hydrodynamics}
\author[A. Ilyin, A. Kostianko, and S. Zelik] {Alexei Ilyin${}^{1,2}$,
Anna Kostianko${}^{4,5}$,
and Sergey Zelik${}^{1,3,4}$}

\subjclass[2000]{35B40, 35B45, 35L70}

\keywords{Lieb--Thirring inequalities, Navier--Stokes equations, attractors, fractal dimension, alpha models}

\thanks{This work was supported by the Russian Science Foundation grant No.19-71-30004.
 The research of the first author is also
supported by Sirius University of Science and Technology
(project `Spectral and Functional Inequalities of Mathematical
Physics and Their Applications'). The second  author was  partially
supported  by the Leverhulme grant No. RPG-2021-072 (United
Kingdom).}

\email{ilyin@keldysh.ru}
\email{a.kostianko@imperial.ac.uk}
\email{s.zelik@surrey.ac.uk}
\address{${}^1$ Keldysh Institute of Applied Mathematics, Moscow, Russia}
\address{${}^2$ Sirius Mathematics Center, Sirius University of Science and Technology,
 Russia, 354349 Sochi, Olimpiyskiy ave. b.1}
\address{${}^3$ University of Surrey, Department of Mathematics, Guildford, GU2 7XH, United Kingdom.}
\address{${}^4$ \phantom{e}School of Mathematics and Statistics, Lanzhou University, Lanzhou\\ 730000,
P.R. China}
\address{${}^5$  Imperial College, London SW7 2AZ, United Kingdom}

\begin{abstract}
We discuss the   estimates for the $L^p$-norms of systems
of functions that are orthonormal in $L^2$ and $H^1$, respectively,
and their essential role in deriving good or even optimal bounds
for the dimension of global attractors for the classical Navier--Stokes equations
and for a class of  $\alpha$-models approximating them. New applications to interpolation inequalities
on the 2D torus are also given.
\end{abstract}

\maketitle

\setcounter{equation}{0}
\section{Introduction}\label{S:Intro}
The 2D Navier--Stokes system is probably one of the
widest known and popular example of an evolution dissipative PDE
possessing a global attractor in an appropriate phase space.
Furthermore, many concepts and ideas of the theory of
infinite dimensional dissipative dynamical systems have originated and have
been developed from this example (see, for instance,~\cite{B-V, T} and
the references therein).

The global attractor is a compact, strictly invariant and globally attracting
set in the phase space, and one of main achievements of the theory
was the proof that its Hausdorff and fractal dimension are finite.
Then followed exponential and afterwards  polynomial estimates of its dimension,
which have saturated (see \cite{B-V83, C-F85}) at the level of
\begin{equation}\label{1}
\dim\mathscr A\le \mathrm{const}\,G^2, \quad G:=\frac{\|f\||\Omega|}{\nu^2},
\end{equation}
where $\|f\|$ is the $L^2$-norm of the forcing term, $|\Omega|$ is the area
of the spatial domain,  $\nu$ is the viscosity coefficient (see \eqref{NSs}),
and the dimensionless number $G$ built of the physical parameters of the system
is called the Grashof number.

The idea to use  Lieb--Thirring inequalities \cite{LT} for $L^2$-orthonormal families  in the study of
attractors of the Navier--Stokes equations was first suggested
by D. Ruelle \cite{Ruelle} and some conjectures of \cite{Ruelle} have been
proved by E. Lieb  \cite{Lieb}. For the two dimensional Navier--Stokes
system in a bounded domain with non slip boundary conditions
R. Temam \cite{T85} using this technique obtained  upper bounds
for the  Hausdorff and fractal dimension of the  attractor  which are
linear with respect to the  Grashof number and
are probably optimal:
\begin{equation}\label{2}
\dim\mathscr A\le  \mathrm{const}\,G.
\end{equation}
At least in terms of the physical parameters this upper bound
stays  unchanged for almost four decades, no lower bounds for the
dimension in the case of the  Dirichlet boundary conditions  are available either.

On the other hand again  for more than four decades  the theory of
the Lieb--Thirring inequalities is still a very active and
dynamically developing area of functional analysis and mathematical
physics. A current state of the art of many aspects of the theory
is presented in \cite{lthbook}.

In \S\,\ref{S:LTLiYau} we formulate  the required Lieb--Thirring
inequality for divergence free $\mathbf{L}^2$-orthonormal vector
functions in two dimensions and also the relevant Li--Yau-type
lower bound for the eigenvalues of the Stokes operator. Then in
\S\,\ref{S:LTNS} we describe in reasonable detail the proof  of
upper bounds \eqref{1} and \eqref{2} and single out  the point,
where the Lieb--Thirring inequality plays the vital role in going
from one to the other.

While the first part of this work is essentially a brief review
of the role of the Lieb--Thirring inequalities in the Navier--Stokes theory,
 the second  part contains new results.  We turn here  from the classical
models in hydrodynamics to a class of their approximations in terms
of  $\alpha$-models. Alpha models became popular over the last
decades both in theory and in practice as subgrid scale models of
turbulence. One of the characteristic features of these models is
the smoothing of the velocity vector $u$ in certain parts of the
bilinear convective term by replacing it with $\bar u:=(1-\alpha\Delta)^{-1}u$,
where $\alpha=\alpha'L^2$, $L$ is the characteristic length,
 and $\alpha'$ is a small dimensionless parameter.

 We also observe that in certain cases
 the energy space is not necessarily $\mathbf{L}^2$. For instance,
in the Euler--Bardina model that we have been interested in
 the natural phase space is $\mathbf{H}^1$
with scalar product
\begin{equation}\label{scalalpha}
(u,v)_\alpha=(u,v)+\alpha(\nabla u,\nabla v).
\end{equation}
The corresponding global Lyapunov exponents
are also estimated in $\mathbf{H}^1$ and in this way
we are led to find bounds for
$\|\rho\|_{L^2}$, where
$$\rho(x):=\sum_{j=1}^n|v_j(x)|^2,
$$
where $\{v_j\}_{j=1}^n$ is an orthonormal family in $\mathbf{H}^1$
with respect scalar product~\eqref{scalalpha}. This type of inequalities
were discovered by E.~Lieb in~\cite{LiebJFA} and remarkably nicely fit
in the estimates we required in \cite{IZLap70, arxiv}
(namely, the $L^2$-bound for $\rho$), where
we have also given explicit expressions for the constants
on $\mathbb{T}^2$ and $\mathbb{T}^3$ to be able to write down explicitly the
estimate for fractal dimension of the global attractor.

The main result here (Theorem~\ref{Th:main}) gives explicit bounds
for the $L^p$-norm on the torus $\mathbb{T}^2$
of the function $\rho$ for all $1\le p<\infty$.

The one-function corollary  of this theorem is equivalent to the
interpolation inequality for $\varphi \in  \dot{H}^1(\mathbb{T}^2)$:
$$
\|\varphi\|_{L^q(\mathbb{T}^2)}\le\left(\frac1{4\pi}\right)^\frac{q-2}{2q}\left(\frac q2\right)^{1/2}
\|\varphi\|^{2/q}\|\nabla\varphi\|^{1-2/q},\qquad q\ge2,
$$
and the constant here should be compared with that in the corresponding inequality
in $\mathbb{R}^2$, see~\eqref{Gag-Nir-R2}.

In \S\,\ref{S:App} we prove the key inequality
for the 2D lattice sum
$$
I_p(m):=\frac{(p-1)m^{2(p-1)}}\pi\sum_{n\in{\mathbb Z}_0^2}\frac1{(m^2+|n|^2)^p}<1,
$$
which was previously proved for $p=2$ in \cite{IZLap70}.
It is easy to see that $I_p(\infty)=1$ so it suffices to
establish monotonicity of $I_p(m)$.
By using a special representation of $I_p(m)$ in terms of the Jacobi
theta function $\theta_3$ the required monotonicity is
proved by showing that
 $\frac d{dm}I_p(m)>0$. We point out that this approach
 works simultaneously for all $p>1$ (and with minor changes on
 $\mathbb{T}^3$ as well).

We finally observe that the monotonicity so obtained is a subtle
property of $I_p(m)$. There are quite a few examples of
similar lattice sums and series with respect to the eigenvalues
of the Laplacian  on the sphere \cite{BDZel-Edinb, IZ} where the
corresponding functions exhibit oscillations for $m$ not too large.

\setcounter{equation}{0}
\section{Lieb--Thirring and Li--Yau-type inequalities for divergence free
orthonormal families}\label{S:LTLiYau}

In this section formulate the Lieb--Thirring inequality
for $\mathbf{L}^2$-orthonormal families of divergence free
vector functions in 2D and the Li--Yau-type lower for the
eigenvalues of the Stokes operator.

\begin{theorem}\label{T:Stokes} {(See \cite{Il_Stokes}.)}
Let $\Omega$ be an arbitrary domain in $\mathbb{R}^d$
with finite volume $|\Omega|<~\infty$. Let a family of
vector functions $\{u_k\}_{k=1}^m\in \mathbf{H}^1_0(\Omega)$ be orthonormal,
and, further,  let  $\div u_k=0$, $k=1,\dots,m$. Then
\begin{equation}\label{lowerorth}
\sum_{k=1}^m\|\nabla u_k\|^2\,\ge\,
\frac d{2+d}\left(
\frac{(2\pi)^d}{\omega_d(d-1)|\Omega|}
\right)^{2/d}m^{1+2/d}\,.
\end{equation}
\end{theorem}

If we take for $u_k$ the first $m$ eigenfunctions of the Stokes operator
\begin{equation}\label{Stokessmooth}
\aligned
&-\Delta u_k + \nabla p_k\,=\,\lambda_ku_k,\\
&\div u_k\,=\,0,\,\,\,u_k\vert_{\partial\Omega }\,=\,0,
\endaligned
\end{equation}
then the left-hand side in~\eqref{lowerorth} becomes $\sum_{k=1}^m\lambda_k$
and in view of the asymptotic formula
(see~\cite{Metiv} at least when $\partial\Omega$ is Lipschitz))
$$
\lambda_k\sim \left(
\frac{(2\pi)^d}{\omega_d(d-1)|\Omega|}
\right)^{2/d}k^{2/d}
$$
we see that the constant on the right-hand side in
\eqref{lowerorth} is sharp in the sense that it cannot be taken greater
rendering \eqref{lowerorth} valid for \emph{all} $m$.

In the 2D case of our main concern we obtain
\begin{equation}\label{Stokes2d}
\sum_{k=1}^m\|\nabla u_k\|^2\,\ge\,\frac{2\pi}{|\Omega|}m^2\quad\text{and}\quad\lambda_1\ge\frac{2\pi}{|\Omega|}\,.
\end{equation}

We also point out that as is shown in~\cite{Kelliher}
in the 2D case for all $k\ge1$
$$
\lambda_k>\lambda^{\mathrm{D}}_k,
$$
where $\lambda^{\mathrm{D}}_k$ are the eigenvalues of the Dirichlet Laplacian.

The next result is crucial in finding
good estimates for the dimension of attractors of the 2D
Navier--Stokes system.

\begin{theorem}\label{T:LT}
Let $\Omega\subseteq\mathbb{R}^2$ be an arbitrary domain. Let a family of
scalar functions $\{\varphi_k\}_{k=1}^m\in {H}^1_0(\Omega)$ be orthonormal
in $L^2(\Omega)$. Then
$$
\rho(x):=\sum_{k=1}^m|\varphi_k(x)|^2
$$
satisfies the inequality
\begin{equation}\label{LTscal}
\int_{\Omega}\rho(x)^2\,dx\le \mathrm{c_{LT}}\sum_{k=1}^m\|\nabla\varphi_k\|^2.
\end{equation}
Let now a family of divergence free vector functions
$\{u_k\}_{k=1}^m\in \mathbf{H}^1_0(\Omega)$,
 $\div u_k=0$, be orthonormal in $\mathbf{L}^2(\Omega)$. Then
$\rho(x):=\sum_{j=1}^m|u_k(x)|^2$ satisfies
\begin{equation}\label{LTvec}
\int_{\Omega}\rho(x)^2\,dx\le \vec{\mathrm{c}}_{\mathrm{LT}}\sum_{k=1}^m\|\nabla u_k\|^2,
\end{equation}
where
\begin{equation}\label{cc}
\vec{\mathrm{c}}_{\mathrm{LT}}\le\mathrm{c_{LT}}.
\end{equation}
\end{theorem}

The constant $\mathrm{c_{LT}}$ is bounded from below by its `classical'
value, which is $1/(2\pi)$, and it is now customary~\cite{lthbook} to write estimates
for it in the form
$$
\mathrm{c_{LT}}\le R\cdot\frac1{2\pi}\,.
$$
Inequality \eqref{LTscal} was originally proved in~\cite{LT} with
$R=3\pi$, followed by significant improvements in~\cite{HLW}, $R=2$,
and in~\cite{DLL}, $R=\pi/\sqrt{3}=1.8138\dots$. The best to date estimate
obtained in \cite{FHJN} is
$$
R=1.456\dots\,.
$$

Finally, it was shown in~\cite{Ch-I} that in two dimensions the constant in the Lieb--Thirring
inequality does not increase in going over from the scalar case to the
divergence free vector case.

\setcounter{equation}{0}
\section{Lieb--Thirring inequalities and attractors for
Navier--Stokes equations}\label{S:LTNS}

We now consider  the two-dimen\-sional Navier--Stokes system
\begin{equation}\label{NSs}
\aligned
&\partial_tu\,+\,\sum_{i=1}^2u^i\partial_iu\,=\,\nu\Delta \,u\,-\,
\nabla\,p\,+\,f,\\
&\operatorname{div} u\,=\,0,\quad
u\vert_{\partial\Omega}\,=\,0,\quad u(0)\,=\,u_0,\qquad\Omega\Subset\R^2
\endaligned
\end{equation}
in a domain $\Omega$ with finite area $|\Omega|<\infty$ and with
Dirichlet boundary conditions for the velocity vector $u$.

We denote by  $\mathrm{P}$ the Helmholtz--Leray orthogonal projection in $\mathbf{L}_2(\Omega )$ onto
the Hilbert space $H$ which is the closure in  $\mathbf{L}_2(\Omega )$ of the set
of smooth solenoidal vector functions with compact supports in $\Omega$.
Applying $\mathrm{P}$ and thereby excluding the pressure $p$ we obtain
the evolution equation in $H$
\begin{equation}\label{NSeq}
\partial_t u\,+\,\nu A\,u\,+\,B(u,u)\,=\,f,\quad u(0)=u_0,
\end{equation}
where $A\,=\,-\mathrm{P}\Delta$ is the Stokes operator with eigenvalues
$\lambda_1\le\lambda_2\le\dots$ and $B(u,v)\,=\,\mathrm{P}\bigl(\sum_{i=1}^2u^i\partial_iv\bigr)$
is the bilinear operator satisfying the fundamental orthogonality relation
\begin{equation}\label{Borth}
(B(u,v),v)=0.
\end{equation}

The equation \eqref{NSeq} has a unique solution in $H$, so that the  the
solution semigroup $S(t):H\to H$, $S(t)u_0=u(t)$ of continuous
 operators is well-defined (see, for instance~\cite{Temam1}). The following two a priori estimates are essential in the proof.
Taking the scalar product of \eqref{NSeq} with $u$ and using \eqref{Borth}
we obtain
$$
\aligned
&\partial_t\|u\|^2\,+\,2\nu\|\nabla u\|^2\,\le\,2\|f\|_{-1}\|\nabla u\|\,\le\,
\nu\|\nabla u\|^2\,+\,\nu^{-1}\|f\|_{-1}^2\,\le\\
&\qquad\le\,\nu\|\nabla u\|^2\,+\,(\lambda_1\nu )^{-1}\|f\|^2,
\endaligned
$$
which gives
\begin{equation}\label{2ap}
\aligned
&\partial_t\|u\|^2\,+\,\nu\|\nabla u\|^2\,\le\,(\lambda_1\nu )^{-1}\|f\|^2,\\
&\partial_t\|u\|^2\,+\,\nu\lambda_1\|u\|^2\,\le\,(\lambda_1\nu )^{-1}\|f\|^2.
\endaligned
\end{equation}
Integrating the first inequality~\eqref{2ap} in time we obtain
\begin{equation}\label{enstr-est}
\limsup_{t\to\infty}\sup_{u_0\in\mathscr{A}}
\frac 1t\int_0^t\|\nabla u(\tau)\|^2d\tau\,
\le\,\frac{\|f\|^2}{\lambda_1\nu^2}\,.
\end{equation}
It follows from the second inequality in~\eqref{2ap} that
the ball $B_0$ in $H$ of radius $2(\nu\lambda_1)^{-1}\|f\|$ is an absorbing set for
the semigroup $S(t)$. Furthermore, the set $B_1:=S(1)B_0$ is bounded
in $\mathbf{H}^1_0(\Omega)\cap H$ (see \cite{Lad92}) and therefore compact in $H$.
Hence, the $\omega$-limit set in $H$ of the set $B_1$ is well defined.
This set is the global attractor of the Navier--Stokes system in the phase space $H$.

\begin{definition}\label{Def:Attractor} A set $\mathscr A\subset H$ is a  global attractor of the
semigroup $S(t)$ of continuous operators acting in a Banach space $H$ if
\par
1) $\mathscr A$ is a compact set in $H$;
\par
2) $\mathscr A$ is strictly invariant, i.e., $S(t)\mathscr A=\mathscr A$;
\par
3) $\mathscr A$ attracts the images of all bounded sets in  $H$, i.e. for every bounded
set $B\subset H$ and every neighbourhood $\mathcal O(\mathscr A)$ of the
attractor  there exists $T=T(\mathcal O, B)$ such that
 $
 S(t)B\subset\mathcal O(\mathscr A) \ \text{for all} \ t\ge T.
 $
\end{definition}

Next we consider the Navier--Stokes system linearized on the solution $u(t)$
lying on the attractor and parameterized by the initial point $u_0$:
\begin{equation}\label{var-eq}
\partial_tU=-\nu AU-
B(U,u(t))-B(u(t),U)=:{\mathcal L}(t,u_0)U, \qquad U(0)=\xi.
\end{equation}

We define and estimate the numbers $q(n)$, that is, the sums of the first $n$
global Lyapunov exponents:
\begin{equation}\label{def-q(m)}
q(n):=\limsup_{t\to\infty}\ \sup_{u_0\in {\mathscr A}}\ \
\frac{1}t
\int_0^t \sup_{\{v_j\}_{j=1}^n}\sum_{j=1}^n\bigl({\mathcal L}(\tau,u_0)v_j,v_j\bigr)d\tau,
\end{equation}
where $\{v_j\}_{j=1}^m\in \mathbf{H}^1_0(\Omega)\cap\{\div v=0\}$
is and arbitrary divergence free  $\mathbf{L}^2$-orthonormal system of dimension~$n$
\cite{B-V, C-F85, T}.

The numbers $q(n)$ control  the expansion or contraction
of the $n$-dimensional volumes transported by the variational equation
along the solution lying on the attractor and their role
in the dimension estimates is crucial
(see \cite{B-V, C-F85, T} and  \cite{Ch-I2001, Ch-I} for the Hausdorff and fractal dimension, respectively).
\begin{theorem}\label{T:frac}
Let for an integer $n>0$ $q(n)\ge0$ and $q(n+1)<0$.
Then both the Hausdorff and the fractal dimensions
of $\mathscr A$ satisfy
$$
\dim \mathscr A\le n_{\mathrm{L}}:=n+\frac{q(n)}{q(n)-q(n+1)}\,.
$$
\end{theorem}
\begin{remark}\label{R:conc}
{\rm
If the function $q$ viewed as a function of a continuous
variable is concave (at least near $n$), then it is
geometrically  clear that $n_{\mathrm{L}}\le n^*$,
where $q(n^*)=0$.
}
\end{remark}

Turning to estimating the numbers $q(n)$ we
integrate by parts and
using the key orthogonality relation \eqref{Borth} we obtain
\begin{equation}\label{Trace1}
\aligned
&\sum_{j=1}^n({\mathcal L} (t,u_0)v_j,v_j)\,=\,
-\nu\sum_{j=1}^n\|\nabla v_j\|^2\,-\,\int\sum_{j=1}^n
\sum_{k,i=1}^2
v_j^k\partial_ku^iv_j^idx\,\le\\
-&\nu\sum_{j=1}^n\|\nabla v_j\|^2\,+\, 2^{-1/2}
\int\rho (x)|\nabla u(t,x)|\,dx\,\le\\
-&\nu\sum_{j=1}^n\|\nabla v_j\|^2\,+\,
2^{-1/2}\|\rho\|\|\nabla u\|\,,
\endaligned
\end{equation}
where we used the pointwise inequality (see, \cite{arxiv,Lieb})
$$\biggl|\sum_{k,i=1}^2
v^k\partial_ku^iv^i\biggr|=
|\nabla u\,v\cdot v|\le
2^{-1/2}|\nabla u||v|^2,$$
and where
$$
\rho(x):=\sum_{j=1}^n|v_j(x)|^2.
$$

Prior to the use of the Lieb--Thirring inequalities in the context
of the attractors for the Navier--Stokes equations the
 function $\rho$ was estimated (in a non-optimal way)  by the Ladyzhenskaya inequality
\begin{equation}\label{Ladineq}
\|u\|^4_{L^4}\le \mathrm{c}_{\mathrm{Lad}}\|u\|^2\|\nabla u\|^2,
\end{equation}
 where
 $$
 \mathrm{c}_{\mathrm{Lad}}\le\frac{16}{27\pi},\qquad
 \mathrm{c}_{\mathrm{Lad}}=\frac{1}{\pi\cdot 1.8622\dots}\,
 $$
see, respectively,~\eqref{Gag-Nir-R2}, and \cite{Wein83}, where
 the sharp value of the  constant was found numerically, by calculating the  norm of the ground state solution of the corresponding Euler--Lagrange equation.

Using \eqref{Ladineq} and the fact that the $v_j$'s are normalized (but not using  orthogonality)
we find that
$$
\rho(x)^2=\sum_{i,j=1}^n|v_i(x)^2|^2|v_j(x)^2|^2\le
\frac12\sum_{i,j=1}^n\left(|v_i(x)|^4+|v_j(x)|^4\right)=
n\sum_{j=1}^n|v_j(x)|^4,
$$
so that
\begin{equation}\label{withoutLT}
\|\rho\|^2=\int_{\Omega}\rho(x)^2\,dx\le
n \,\mathrm{c}_{\mathrm{Lad}}\sum_{j=1}^n\|\nabla v_j\|^2.
\end{equation}
Substituting this into~\eqref{Trace1} and  splitting
the second term there accordingly, we obtain
$$
\sum_{j=1}^n({\mathcal L} (t,u_0)v_j,v_j)\le
-\frac\nu 2\sum_{j=1}^n\|\nabla v_j\|^2+
\frac{n\mathrm{c}_{\mathrm{Lad}}}{4\nu}\|\nabla u(t)\|^2.
$$
It remains to use the lower bound for the sums of eigenvalues of the Stokes
operator \eqref{Stokes2d} including the lower bound $\lambda_1\ge2\pi/|\Omega|$
and  estimate~\eqref{enstr-est} for the solutions lying on the attractor.
This finally gives
\begin{equation}\label{qn1}
q(n)\le
-\frac{\nu\pi }{|\Omega|}n^2+
\frac{ \mathrm{c}_{\mathrm{Lad}}\|f\|^2|\Omega|}{8\pi\nu^3}n\,,
\end{equation}
so that $q(n^*)=0$ for
$$
n^*=\frac{\mathrm{c}_{\mathrm{Lad}}}{8\pi^2}G^2, \qquad
G=\frac{\|f\||\Omega|}{\nu^2}\,,
$$
and the number $n^*$ is an upper bound both for the Hausdorff and
fractal dimension of the global attractor $\mathscr A$:
$$
\dim \mathscr A\le\frac{\mathrm{c}_{\mathrm{Lad}}}{8\pi^2}G^2\,.
$$
Up to explicit constants this is what the situation in this area looked like
before the use the Lieb--Thirring inequalities, see, for instance,
\cite{B-V83}, \cite{C-F85}.

The Lieb--Thirring bound for orthonormal families gives an optimal
bound for the function $\rho$ replacing \eqref{withoutLT} by the inequality with
constant  independent of the size $n$ of the orthonormal family:
\begin{equation}\label{orthLT}
\|\rho\|^2=\int_{\Omega}\rho(x)^2\,dx\le
\mathrm{c}_{\mathrm{LT}}\sum_{j=1}^n\|\nabla v_j\|^2.
\end{equation}
Replacing \eqref{withoutLT} with \eqref{orthLT} in \eqref{Trace1}
and arguing as before we find that
\begin{equation}\label{Trace2}
\aligned
&\sum_{j=1}^n({\mathcal L} (t,u_0)v_j,v_j)\,\le\,
-\nu\sum_{j=1}^n\|\nabla v_j\|^2\,+\,
2^{-1/2}\|\rho\|\|\nabla u\|\,\le\\
-&\nu\sum_{j=1}^n\|\nabla v_j\|^2\,+\,
2^{-1/2}
\biggl(\mathrm{c}_{\mathrm{LT}}
\sum_{j=1}^n\|\nabla v_j\|^2\biggr)^{1/2}
\|\nabla u(t)\|\,\le\\
-&\frac\nu2\sum_{j=1}^n\|\nabla v_j\|^2\,+\,
\frac{\mathrm{c}_\mathrm{LT}}{4\nu}\|\nabla u(t)\|^2\,\le\,
-\frac{\nu \pi}{|\Omega|}n^2\,+\,
\frac{c_\mathrm{LT}}{4\nu}\|\nabla u(t)\|^2\,,
\endaligned
\end{equation}
so that~\eqref{qn1} goes over to
\begin{equation}\label{qn2}
q(n)\le
-\frac{\nu\pi }{|\Omega|}n^2+
\frac{ \mathrm{c}_{\mathrm{LT}}\|f\|^2|\Omega|}{8\pi\nu^3}\,,
\end{equation}
which gives the estimate of the dimension that is linear with respect
to the dimensionless number $G$.
\begin{theorem}\label{T:dimNS}
The Hausdorff and the fractal dimension of the global attractor $\mathscr A$
of the Navier--Stokes system in a domain $\Omega\subset\mathbb{R}^2$ with
finite area satisfy the estimate
\begin{equation}\label{dimNS}
\dim \mathscr A\le\frac{\mathrm{c}_{\mathrm{LT}}^{1/2}}{2\sqrt{2}\pi}G\,,
\qquad G=\frac{\|f\||\Omega|}{\nu^2}\,.
\end{equation}
\end{theorem}

\begin{remark}
{\rm
One can avoid  the Li--Yau-type lower bounds for the Stokes operator
by using instead that
$$
n^2=\left(\int_{\Omega}\rho(x)\,dx\right)^2\le|\Omega|\|\rho\|^2.
$$
Then we obtain
$$
\aligned
&\sum_{j=1}^n({\mathcal L} (t,u_0)v_j,v_j)\,\le\,
-\frac{\nu}{\mathrm{c}_{\mathrm{LT}}}\|\rho\|^2\,+\,
2^{-1/2}\|\rho\|\|\nabla u\|\,\le\\
&-\frac{\nu}{2\mathrm{c}_{\mathrm{LT}}}\|\rho\|^2\,+\,
\frac{\mathrm{c}_{\mathrm{LT}}\|\nabla u(t)\|^2}{4\nu}\,\le
-\frac{\nu}{2\mathrm{c}_{\mathrm{LT}}|\Omega|}n^2\,+\,
\frac{\mathrm{c}_{\mathrm{LT}}\|\nabla u(t)\|^2}{4\nu}\,.
\endaligned
$$
Using again \eqref{enstr-est} we obtain
$$
\dim \mathscr A\le\frac{\mathrm{c}_{\mathrm{LT}}}{2\sqrt{\pi}}G\,,
\qquad G=\frac{\|f\||\Omega|}{\nu^2}\,.
$$
However, the factor of $G$ in \eqref{dimNS} is smaller, since
$\mathrm{c_{LT}}\ge 1/(2\pi)$.
}
\end{remark}

 \setcounter{equation}{0}
\section{$L^p$-inequalities for families of functions with
orthonormal derivatives and applications}\label{S:alpha}

We have seen in seen in \S\,\ref{S:LTLiYau} and \S\,\ref{S:LTNS}
that the Lieb--Thirring inequality for $L^2$-orthonormal families
is essential for finding good estimates for the attractor
dimension of the 2D Navier--Stokes system in a bounded domain with Dirichlet
boundary conditions.

In the 3D case the situation is drastically different,
since the global well-posedness
remains a mystery and therefore  inspires a comprehensive study of various
 modifications/regularizations of the initial Navier-Stokes/Euler equations
  (such as various $\alpha$ model, hyperviscous Navier-Stokes equations,
  regularizations via $p$-Laplacian, etc.), many of which have a strong
   physical background and are of independent interest, both in practice and theory,
    see e.g.
   \cite{{BFR80}, Camassa,HLT10, Larios, Lopes,OT07} and the references therein.

   In \cite{arxiv, IZLap70,  MZ} the authors have recently studied
the following regularized damped Euler system:
\begin{equation}\label{DEalpha}
\left\{
  \begin{array}{ll}
    \partial_t u+(\bar u,\nabla)\bar u+\gamma u+\nabla p=g,\  \  \\
    \operatorname{div} \bar u=0,\quad u(0)=u_0.
  \end{array}
\right.
\end{equation}
with  forcing $g$ and  Ekman damping term $\gamma u$, $\gamma>0$.
The damping term $\gamma  u$ makes the system dissipative and is
important in  various geophysical models. System \eqref{DEalpha}
(at least in the conservative case $\gamma=0$) is often referred to
as the  simplified Bardina subgrid scale model of turbulence, see
\cite{BFR80} for the derivation of the model and further
discussion.

The system is studied for $d=2,3$\newline
1) on the torus $\mathbb{T}^d=[0,L]^d$ with
      standard zero mean condition;\newline
2) in  $\Omega\subseteq\mathbb{R}^d$;\newline
3) on the sphere $\mathbb{S}^2$ or in a domain on it  $\Omega\subseteq\mathbb{S}^2$; \newline
4) if $\Omega\varsubsetneq\mathbb{R}^d$ or $\Omega\varsubsetneq\mathbb{S}^2$, then  $\bar u\vert_{\partial\Omega}=0$
 and  $\bar u$ is recovered from $u$ by solving the Stokes problem
   $$
      \left\{\aligned(1-\alpha\Delta)\bar u+\nabla q=u,\\
      \operatorname{div}\bar u=0,\quad \bar u\vert_{\partial\Omega}=0.
      \endaligned\right.
      $$
In the case when there is no boundary
$$
\bar u=(1-\alpha\Delta)^{-1}u.
$$
Here $\alpha=\alpha'L^2$ and $\alpha'>0$ is a small dimensionless parameter, so that
$\bar u$ is a smoothed  vector field with higher spatial modes filtered out.

 The phase space with respect to  $\bar u$ is the Sobolev space $\mathbf{H}^{1}$ with divergence free condition
\begin{equation}\label{H1defin}
\bar u\in\mathbf{H}^1:=\left\{
                       \begin{array}{ll}
                         \dot{\mathbf{H}}^1(\mathbb{T}^d), & x\in\mathbb{T}^d,
\ \int_{\mathbb{T}^d}\bar u(x)dx=0, \\
                         \mathbf{H}^1(\mathbb{R}^d), & x\in\mathbb{R}^d, \\
                       \mathbf{H}^1_0(\Omega), & x\in\Omega\varsubsetneq\mathbb{R}^d,\mathbb{S}^2
                       \end{array}
                     \right.\qquad\operatorname{div}\bar u=0,
\end{equation}
with scalar product~\eqref{scalalpha}.

The results obtained in  \cite{arxiv, IZLap70,  MZ}
can be combined into   the following theorem.

\begin{theorem}
Let $d=2$. In each case of BC the system possesses a global attractor $\mathscr A\Subset\mathbf{H}^1$
with finite fractal dimension satisfying
$$\dim_F{\mathscr A}\le\frac1{8\pi}\cdot\left\{
\aligned
&\frac1{\alpha\gamma^4}
\min\left(\|\operatorname{curl}g\|^2_{L^2},\ \frac{\|g\|^2_{L^2}}{2\alpha}\right),
\quad x\in\mathbb{T}^2, \mathbb{R}^2, \mathbb{S}^2\\
&\frac{\|g\|^2_{L^2}}{2\alpha^2\gamma^4},\quad x\in\Omega\varsubsetneq\mathbb{R}^2,\mathbb{S}^2.
\endaligned
\right.
$$
In the 3D case the estimates in all tree cases look formally the same
$$
\dim_F\mathscr{A}\le
\frac{1}{12\pi}\frac{\| g\|_{L^2}^2}{\alpha^{5/2}\gamma^4},
\quad x\in\mathbb{T}^3,\  x\in\mathbb{R}^3,\  x\in\Omega\varsubsetneq\mathbb{R}^3\,.
$$
Furthermore, in the periodic case both on $\mathbb{T}^2$,
and $\mathbb{T}^3$ the upper bounds are optimal
in the limit as  $\alpha\to0^+$.
\end{theorem}

To the best of our knowledge this is the first example of a meaningful  3D
hydrodynamic model with sharp two-sided estimates for
the dimension of the global attractor.

Optimal lower bounds for the torus $\mathbb{T}^2$ are based on the instability
analysis of the specific stationary solutions --- generalized
Kolmogorov flows \cite{IZLap70} and are carried over to $\mathbb{T}^3$
by means of the Squire's transformation \cite{arxiv}.

Explicit upper bounds for all types of domains and boundary
conditions as  for the Navier--Stokes system are obtained by
finding good estimates for the global Lyapunov exponents.  However,
the phase space is now $\mathbf{H}^1$ and the $n$-traces of the
linearized operator are now calculated with respect to the scalar
product~\eqref{scalalpha}. The Lieb--Thirring inequalities for
$\mathbf{L}^2$-orthonormal families are replaced by inequalities
for the  $L^p$-forms of families of functions with orthonormal
derivatives also obtained  by E. Lieb in ~\cite{LiebJFA}.

\begin{proposition}\label{P:L2bounds} \cite{arxiv}.
Let $\{v_j\}_{j=1}^n\in\mathbf{H}^1$, see~\eqref{H1defin}. Suppose that
this family is orthonormal with respect to the scalar product~\eqref{scalalpha}:
\begin{equation}\label{alpha}
(v_i,v_j)+\alpha(\nabla u_i,\nabla v_j)=\delta_{ij}.
\end{equation}
Then the function
$$
\rho(x)=\sum_{j=1}^n|v_j(x)|^2
$$
satisfies
\begin{equation}\label{L2bounds}
\aligned
\aligned
&\|\rho\|_{L^2}\le\frac1{2\sqrt{\pi}}\frac{n^{1/2}}{\alpha^{1/2}},\qquad d=2,\\
&\|\rho\|_{L^2}\le\frac1{2\sqrt{\pi}}\frac{n^{1/2}}{\alpha^{3/4}}, \qquad d=3.
\endaligned
\endaligned
\end{equation}
\end{proposition}

Inequalities of this type in $\mathbb{R}^d$ with $L^p$-norm on the left-hand side were proved  in~\cite{LiebJFA}
for   $p=\infty$ ($d=1$),  $1\le p<\infty$ ($d=2$), and for the critical
$p=d/(d-2)$ ($d\ge3$). No expressions for the constants were given. Our interest in what follows
is in the case of the 2D torus $\mathbb{T}^2$. Furthermore, since the scalar
product~\eqref{alpha} is defined by  system~\eqref{DEalpha}, and we are
now interested just in inequalities themselves, we turn to the more convenient
(and clearly equivalent) scalar product
\eqref{orth-m} which is also used in~\cite{LiebJFA}.
Finally, we consider the scalar case only, since in $x$-coordinates
the vector  case involves no problems at all,
while the case of the sphere $\mathbb{S}^2$ will  be
treated  in a forthcoming work.

\begin{theorem}\label{Th:main}
Let $\{\varphi_j\}_{j=1}^n$ be a family of zero mean functions on
the torus $\{\varphi_j\}_{j=1}^n\in\dot{H}^1(\mathbb{T}^2)$ or let
$\{\varphi_j\}_{j=1}^n\in {H}^1_0(\Omega)$, where
$\Omega\subseteq\mathbb{R}^2$ is an arbitrary domain. Suppose further that
in either case the family is
 orthonormal with respect to the scalar product:
\begin{equation}\label{orth-m}
m^2(\varphi_i,\varphi_j)+(\nabla\varphi_i,\nabla\varphi_j)=\delta_{ij}.
\end{equation}
Then for $1\le p<\infty$ the function
$$
\rho(x):=\sum_{j=1}^n|\varphi_j(x)|^2
$$
satisfies the inequality
\begin{equation}\label{Liebd2}
\|\rho\|_{L^p}\le\mathrm{B}_pm^{-2/p}n^{1/p},
\end{equation}
where
\begin{equation}\label{Liebbp}
\mathrm{B}_p\le\left(\frac{p-1}{4\pi}\right)^{(p-1)/p}.
\end{equation}
\end{theorem}
\begin{proof}
Since inequality \eqref{Liebd2} with constant \eqref{Liebbp} clearly holds for $p=1$,
we assume below that $1<p<\infty$. We also first consider the periodic case.

 Let us define two operators
$$
 \mathbb{H}= V^{1/2}(m^2-{\Delta})^{-1/2}\Pi,\quad
  \mathbb{H}^*=\Pi(m^2-{\Delta})^{-1/2}V^{1/2},
$$
 where $V\in L^p$ is a
non-negative scalar function  and $\Pi$ is the
projection onto the space of functions with mean value zero:
$$
\Pi\varphi=\varphi-\frac1{4\pi^2}\int_{\mathbf{T}^2}\varphi(x)\,dx.
$$
 Then ${\bf K}=
\mathbb{H}^*\mathbb{H}$ is a  compact self-adjoint operator in
 ${L}^2({\mathbb{T}}^2)$  and for $r=p'=p/(p-1)\in(1,\infty)$
$$
\aligned
\Tr \mathbf{K}^r=\Tr\left(\Pi(m^2-{\Delta})^{-1/2}V(m^2-{\Delta})^{-1/2}\Pi\right)^r\le\\\le
\Tr\left(\Pi(m^2-{\Delta})^{-r/2}V^r(m^2-{\Delta})^{-r/2}\Pi\right)=\\=
\Tr\left(V^r(m^2-{\Delta})^{-r}\Pi\right),
\endaligned
$$
where we used
 the Araki--Lieb--Thirring inequality for traces \cite{Araki, LT, traceSimon}:
$$
\Tr(BA^2B)^p\le\Tr(B^pA^{2p}B^p),\quad p\ge1,
$$
and the cyclicity property of the trace together with the facts
that $\Pi$ commutes with the Laplacian and that $\Pi$ is a
projection: $\Pi^2=\Pi$. Using the basis of orthonormal
eigenfunctions of the Laplacian
$$\frac1{2\pi}e^{in\cdot x},\quad
n\in\mathbb{Z}^2_0=\mathbb{Z}^2\setminus\{0,0\}$$
in view of the key estimate~\eqref{0.pm} proved  below we
find that
$$
\aligned
\operatorname{Tr} \mathbf{K}^r\le&
\operatorname{Tr}\left(V^r(m^2-{\Delta})^{-r}\Pi\right)=\\=&
\frac1{4\pi^2}\sum_{n\in\mathbb{Z}^2_0}\frac{1}{\bigl(m^2+|n|^2\bigr)^r}
\int_{\mathbb{T}^2}V^r(x)\,dx\le
\frac1{4\pi}\frac{m^{-2(r-1)}}{{r-1}}\|V\|^r_{L^r},.
\endaligned
$$
We can now argue as in~\cite{LiebJFA}. We observe that
$$
\int_{\mathbb{T}^2}\rho(x)V(x)\,dx=\sum_{i=1}^n\|\mathbb{H}\psi_i\|^2_{L^2},
$$
where
$$
\psi_j=(m^2-{\Delta})^{1/2}\varphi_j,\quad j=1,\dots,n.
$$
Next, in view of \eqref{orth-m} the $\psi_j$'s are
orthonormal in $L^2$
and in view of  the variational
principle
$$
\sum_{i=1}^n\|\mathbb{H}\psi_i\|^2_{L^2}=\sum_{i=1}^n(\mathbf{K}\psi_i,\psi_i)
\le\sum_{i=1}^n\lambda_i,
$$
where $\lambda_i$ are the eigenvalues of the
operator $\mathbf{K}$. Therefore
$$
\aligned
\int_{\mathbb{T}^2}\rho(x)V(x)\,dx\le\sum_{j=1}^n\lambda_j\le
n^{1/p}\left(\Tr K^r\right)^{1/r}\le\\\le  n^{1/p}
\left(\frac{p-1}{4\pi m^{2/(p-1)}}\right)^{(p-1)/p}\|V\|_{L^{p/(p-1)}}=\\=
n^{1/p}m^{-2/p}\left(\frac{p-1}{4\pi}\right)^{(p-1)/p}\|V\|_{L^{p/(p-1)}}.
\endaligned
$$
Finally, setting $V(x)=\rho(x)^{p-1}$ we obtain~\eqref{Liebd2},
\eqref{Liebbp}. This completes the proof for the torus.

The proof of the theorem for $\Omega=\mathbb{R}^2$ is a word for word
repetition of the above proof(without the projection $\Pi$, of course) and
with use of the Fourier transform instead of the Fourier series.
Furthermore, instead of a non-trivial inequality \eqref{0.pm}
for the Green's function on the diagonal we have the
equality
$$
\frac{(p-1)m^{2(p-1)}}\pi\int_{\mathbb{R}^2}\frac{dx}{(m^2+|x|^2)^p}=1.
$$

Finally, if $\Omega\varsubsetneq \mathbb{R}^2$ is a proper domain,
we extend by zero the vector functions $\varphi_j$ outside
$\Omega$ and denote the results by $\widetilde{\varphi}_j$, so
that $\widetilde{\varphi}_j\in {H}^1(\mathbb{R}^2)$. We further set
$\widetilde\rho(x):=\sum_{j=1}^n|\widetilde{\varphi}_j(x)|^2$.
Then setting in $\mathbb{R}^2$
$\widetilde\psi_i:=(m^2-{\Delta})^{1/2}\widetilde{\varphi}_i$,
we see that the system $\{\widetilde\psi_j\}_{j=1}^n$ is
orthonormal in ${L}^2(\R^2)$. Since clearly
$\|\widetilde\rho\|_{L^p(\R^d)}=\|\rho\|_{L^p(\Omega)}$, the proof
now reduces to the case  $\Omega=\mathbb{R}^2$ and therefore
is complete.
\end{proof}

\begin{corollary}\label{Cor:1_func}
The following interpolation inequality holds:
\begin{equation}\label{Gag-Nir}
\|\varphi\|_{L^q}\le\left(\frac1{4\pi}\right)^\frac{q-2}{2q}\left(\frac q2\right)^{1/2}
\|\varphi\|^{2/q}\|\nabla\varphi\|^{1-2/q},\qquad q\ge2.
\end{equation}
\end{corollary}
\begin{proof}
 For $n=1$ inequality \eqref{Liebd2} goes over to
$$
\|\varphi\|_{L^{2p}}\le\mathrm{B}_p^{1/2}m^{-1/p}
\left(m^2\|\varphi\|^2+\|\nabla\varphi\|^2\right)^{1/2},\quad
p\ge 1.
$$
Furthermore, writing this in the form
\begin{equation}
\label{additive}
\|\varphi\|_{L^{2p}}^2\le\mathrm{B}_p
\left(m^{2-2/p}\|\varphi\|^2+m^{-2/p}\|\nabla\varphi\|^2\right)
\end{equation}
and minimizing with respect  $m$ we obtain
\begin{equation}\label{multip}
\aligned
\|\varphi\|_{L^{2p}}^2\le\mathrm{B}_p\frac p{(p-1)^{(p-1)/p}}
\|\varphi\|^{2/p}\|\nabla\varphi\|^{2-2/p}=\\=
\left(\frac1{4\pi}\right)^{(p-1)/p}\,p\,
\|\varphi\|^{2/p}\|\nabla\varphi\|^{2-2/p},
\endaligned
\end{equation}
which is~\eqref{Gag-Nir}.
\end{proof}

The one function inequality \eqref{additive} for the torus $\mathbb{T}^2$
 and the equivalent multiplicative
inequality \eqref{multip}
can be proved in a more direct way in  which, however,  estimate \eqref{0.pm} as before
 plays the essential role. For the case of $\mathbb{R}^2$, see Remark~\ref{R:Bab-Beck}.
\begin{proof}[Direct proof of {Corollary}~\ref{Cor:1_func} for the torus]
 In fact, writing a zero mean
function $\varphi$ on the torus $\mathbb{T}^2$ in terms of the Fourier series
$$
\varphi(x)=\sum_{n\in\mathbb{Z}^2_0}a_ne^{ix\cdot n}
$$
we have by the Parseval identity
$$
\|\varphi\|_{L^2}=2\pi\|a\|_{l^2},
$$
and, furthermore, since the exponentials $e^{ix\cdot n}$ have norm
$1$ in $L^\infty$, we have
$$
\|\varphi\|_{L^\infty}\le\|a\|_{l^1}.
$$
This gives by the Riesz--Thorin interpolation theorem the well-known
Hausdorff--Young inequality
$$
\|\varphi\|_{L^p}\le(2\pi)^{2/p}\|a\|_{l^q},\quad
\frac1p+\frac1q=1,\quad p\ge2.
$$
Thus, by H\"older's inequality for an arbitrary $m>0$
\begin{equation}\label{mult-div}
\aligned
\|\varphi\|_{L^p}\le(2\pi)^{2/p}\|a\|_{l^q}=
(2\pi)^{2/p}\bigl\|(m^2+|n|^2)^{-1/2}\cdot(m^2+|n|^2)^{1/2}a_n\bigr\|_{l^q}\le\\
\le(2\pi)^{2/p}
\biggl(\sum_{n\in\mathbb{Z}^2_0}\frac1{(m^2+|n|^2)^{r/2}}\biggr)^{1/r}
\|(m^2+|n|^2)^{1/2}|a_n|\|_{l^2},
\endaligned
\end{equation}
where $\frac1r+\frac12=\frac1q$, so that
$$
\frac1r=\frac12-\frac1p=\frac{p-2}{2p},\qquad\frac r2-1=\frac2{p-2}.
$$
We now use the key inequality~\eqref{0.pm} in \eqref{mult-div}
$$
\biggl(\sum_{n\in\mathbb{Z}^2_0}\frac1{(m^2+|n|^2)^{r/2}}\biggr)^{1/r}<
\biggl(\frac\pi{(r/2-1)m^{2(r/2-1)}}\biggr)^{1/r}=
\left(\frac{\pi(p-2)}2\right)^{\frac{p-2}{2p}}m^{-\frac2p}
$$
and take into account  that
$$
\|(m^2+|n|^2)^{1/2}|a_n|\|^2_{l^2}=\frac1{4\pi^2}\|(m^2-\Delta)^{1/2}\varphi\|^2=
\frac1{4\pi^2}(m^2\|\varphi\|_{L^2}^2+\|\nabla \varphi\|^2_{L^2}).
$$
We obtain
$$
\|\varphi\|_{L^p}\le\left(\frac{p-2}{8\pi}\right)^{\frac{p-2}{2p}}
m^{-2/p}\left(m^2\|\varphi\|_{L^2}^2+\|\nabla \varphi\|^2_{L^2}\right)^{1/2},
\quad p\ge 2.
$$
Taking the square and changing $p$ to $2p$ gives the inequality
$$
\|\varphi\|^2_{L^{2p}}\le\left(\frac{p-1}{4\pi}\right)^{\frac{p-1}p}
\left(m^{2-2/p}\|\varphi\|^2+m^{-2/p}\|\nabla\varphi\|^2\right),\quad p\ge1,
$$
which \emph{coincides} with \eqref{additive} and is equivalent to
\eqref{multip}.
\end{proof}

\begin{remark}
{\rm
 For $q=4$ in \eqref{Gag-Nir}, that is, in the Ladyzhenskaya
inequality on the 2D torus $\mathbb{T}^2$ the constant is $1/\sqrt{\pi}$ and should be compared
with (and is greater than) the recent estimate of it
 as a one particle Lieb--Thirring
inequality~\cite{ILZ-JFA}
$$
\frac1\pi>\frac{3\pi}{32}\,.
$$
On the other hand, \eqref{Gag-Nir} works for all $q\ge2$ and
provides a simple expression for the constant.
}
\end{remark}

\begin{remark}\label{R:Bab-Beck}
{\rm It is worth mentioning that
the similar approach plus the knowledge
of the sharp Babenko--Beckner inequality  for the Fourier transform
$$
\| f\|_{L^p(\mathbb{R}^d)}\le
\left((2\pi)^{\frac1p-\frac1q}\ \frac{q^\frac1q}{p^\frac1p}\right)^{d/2}
\|\widehat f\|_{L^q(\mathbb{R}^d)}, \quad p\ge2,\quad \frac1p+\frac1q=1
$$
in the analog of \eqref{mult-div} for $\mathbb{R}^2$
gives the following improvement of inequality \eqref{Gag-Nir}
for $\mathbb{R}^2$ with the best to date closed form  estimate for the constant \cite{Nasibov}:
\begin{equation}\label{Gag-Nir-R2}
\|\varphi\|_{L^q(\mathbb{R}^2)}\le\left(\frac1{4\pi}\right)^\frac{q-2}{2q}
\frac{q^{(q-2)/q}}{(q-1)^{(q-1)/q}}
\left(\frac q2\right)^{1/2}
\|\varphi\|^{2/q}\|\nabla\varphi\|^{1-2/q},\ q\ge2,
\end{equation}
where  in comparison with \eqref{Gag-Nir} the middle factor in the constant here is
due to the Babenko--Beckner inequality and is
less than~$1$ for $q\in (2,\infty)$;
see also \cite[Theorem 8.5]{Lieb--Loss} where the  equivalent result is obtained
for the inequality in the additive form.
}
\end{remark}

\begin{remark}\label{R:Orlitz}
{\rm
The rate of growth as $q\to\infty$ of the constant both in \eqref{Gag-Nir} and \eqref{Gag-Nir-R2},
namely $q^{1/2}$, is optimal in the power scale, since otherwise the Sobolev space
$H^1$ in two dimensions would have been embedded in the Orlicz space with Orlicz function
$e^{t^{2+\varepsilon}}-1$, $\varepsilon>0$, which is impossible \cite{Tr}.
}
\end{remark}

\setcounter{equation}{0}
\section{Appendix. Monotonicity of lattice sums}\label{S:App}

In this section we prove two key estimates for the lattice sums in
dimension $2$ and $3$.

\begin{proposition}\label{main0}
The following inequality holds for  $p>1$ and all $m\ge0$
\begin{equation}\label{0.pm}
I_p(m):=\frac{(p-1)m^{2(p-1)}}\pi\sum_{n\in{\mathbb Z}_0^2}\frac1{(m^2+|n|^2)^p}<1.
\end{equation}
\end{proposition}
\begin{proof}
Inequality~\eqref{0.pm} will obviously follow if we show that
\begin{equation}\label{liminfty}
\lim_{m\to\infty}I_p(m)=1
\end{equation}
and $I_p(m)$ is monotone increasing:
\begin{equation}\label{monotone}
\frac d{dm}I_p(m)>0\quad\text{for}\ m>0.
\end{equation}

The proof of \eqref{liminfty} is easy.
Setting
$$
f(x):=\frac1{(1+|x|^2)^p},\quad x\in\mathbb R^2
$$
we use the Poisson summation formula and write:
$$
\aligned
\sum_{n\in{\mathbb Z}_0^2}\frac1{(m^2+|n|^2)^p}=m^{-2p}
\left(\sum_{n\in{\mathbb Z}^2}f(|n|/m)-1\right)=\\
m^{-2p}\left(m^2\int_{{\mathbb R}^2}f(x)\,dx-1
+2\pi m^2\sum_{n\in{\mathbb Z}^2_0}\widehat f(2\pi|n|m)\right)=\\
\frac1{m^{2(p-1)}}\frac\pi{p-1}-\frac1{m^{2p}}+\frac{2\pi}{m^{2(p-1)}}
\sum_{n\in{\mathbb Z}_0^2}\widehat f(2\pi|n|m),
\endaligned
$$
where $\widehat f(\xi)$ is the Fourier transform of $f$. Since the function $f(z)$
is analytic in the strip $|\mathrm{Im}z|<1$, it follows that its Fourier transform
is exponentially decaying and therefore the sum of the last two terms
is negative for all sufficiently large $m$ and is of the order $O(1/m^{2p})$ as
$m\to\infty$. This proves~\eqref{liminfty}.

 More precisely, since $f$ is radial
$$
\widehat f(\xi)=\frac1{2\pi}\int_{\mathbb R^2}f(x)e^{i\xi\cdot x}\,dx=g(t), \quad t=|\xi|,
$$
and where
$$
g(t)=\int_0^\infty\frac{J_0(tr)r\,dr}{(1+r^2)^p}=\frac1{2^{p-1}\Gamma(p)}\,t^{p-1}K_{p-1}(t).
$$
Here  $K_\nu$ is the modified Bessel function of the second kind,
and  the second equality is formula $13.51\,(4)$ in \cite{Watson}.
This finally gives
$$
I_p(m)=1-\frac{p-1}\pi\frac1{m^2}+\frac{4(p-1)}{2^p\Gamma(p)}\sum_{n\in\mathbb Z^2_0}F_{p-1}(2\pi|n|m),
$$
where
$$
F_{p}(t):=t^{p}K_{p}(t).
$$
It remains to recall that
$$
K_p(t)=\sqrt{\frac2\pi}\frac{e^{-t}}{\sqrt{t}}\left(1+O\left(\frac1t\right)\right)\quad
\text{as}\ t\to \infty.
$$

We now turn to the proof of~\eqref{monotone}. Using the
formula
$$
M^{-p}=\frac1{\Gamma(p)}\int_0^\infty x^{p-1}e^{-Mx}\,dx
$$
with $M=m^2+n_1^2+n_2^2$ and summing over the lattice $n\in\mathbb{Z}^2_0$ we obtain
\begin{equation}\label{Itheta}
I_p(m)=\frac{(p-1)m^{2(p-1)}}{\pi\Gamma(p)}\int_0^\infty x^{p-1}e^{-m^2x}\bigl(\theta_3^2(e^{-x})-1\bigr)\,dx,
\end{equation}
where $\theta_3(q)$ is the Jacobi theta function
$$
\theta_3(q)=\sum_{n=-\infty}^\infty q^{n^2}.
$$
Crucial for us is the following functional relation that is a corollary
of the Poisson summation formula and the Fourier transform of the Gaussian
\begin{equation}\label{funcrel}
\varphi(x)=\frac{\varphi(x^{-1})}{\sqrt{x}},\qquad \varphi(x):=\theta_3(e^{-\pi x})=
\sum_{n=-\infty}^\infty  e^{-\pi x n^2}.
\end{equation}
Rewriting \eqref{Itheta} in terms of the function $\varphi$, changing the variable and
then using \eqref{funcrel} we arrive at
$$
\aligned
I_p(m)=\frac{(p-1)}{\pi\Gamma(p)}\int_0^\infty x^{p-1}e^{-x}
\left(m^{-2}\varphi^2\left(\frac{x}{\pi m^2}\right)-m^{-2}\right)\,dx=
\\=\frac{(p-1)}{\pi\Gamma(p)}\int_0^\infty x^{p-1}e^{-x}
\left(\frac\pi{ x}\varphi^2\left(\frac{\pi m^2}{x}\right)-m^{-2}\right)\,dx.
\endaligned
$$
Therefore
$$
\frac d{dm}I_p(m)=
\frac{2(p-1)}{\pi\Gamma(p)m^3}\int_0^\infty x^{p-1}e^{-x}
\left(2y^2\varphi\left(y\right)
\varphi'\left(y\right)+1\right)\,dx,
$$
where $y=y(x):={\pi m^2}/x$ ,
so that the monotonicity will be verified  if we prove the following inequality
\begin{equation}\label{condmon}
2y^2\varphi(y)\varphi'(y)+1\ge0, \quad y\in\mathbb{R}_+.
\end{equation}
It is worthwhile to say that this sufficient condition for monotonicity
is independent both of $m$ and $p$\,!

Since $\varphi(y)$ decays extremely fast as $y$ grows, it is most important to
verify \eqref{condmon} near $y=0$. For this purpose we again use
\eqref{funcrel} and taking into account that
$$
2\varphi(t)\varphi'(t)=(\varphi^2(t))'=-\sum_{n\in\mathbb{Z}^2}
\pi|n|^2e^{-\pi|n|^2t}=-\sum_{n\in\mathbb{Z}^2_0}\pi|n|^2e^{-\pi|n|^2t}
$$
we obtain
$$
\aligned
2y^2\varphi(y)\varphi'(y)=
y^2(\varphi^2(y))'=y^2\left(y^{-1}\varphi^2(y^{-1})\right)'=
-\varphi^2(y^{-1})-\\-2y^{-1}\varphi(y^{-1})\varphi'(y^{-1})=-1+
\sum_{n\in\mathbb Z^2_0}\left(y^{-1}\pi (n_1^2+n_2^2)-1\right)e^{-(n_1^2+n_2^2)\pi y^{-1}}>-1
\endaligned
$$
for $y\le\pi$, so that \eqref{condmon} holds in this case.

Thus, we only need to check \eqref{condmon} for $y\ge\pi$.
We have a lot of free space here and this can be done in many ways, for instance, we may replace
\begin{equation}\label{fipsi}
\varphi(y)=\sum_{n=-\infty}^\infty e^{-\pi n^2y}\le \sum_{n=-\infty}^\infty e^{-|n|\pi y}=
1+\frac{2e^{-\pi y}}{1-e^{-\pi y}}=\coth(\pi y/2)=:\psi(y)
\end{equation}
with the similar estimate for the derivative:
$$
0>\varphi'(y)\ge \psi'(y), \quad y\ge \pi,
$$
which holds in view of the elementary inequality
$$
ne^{-an^2}\le e^{-an}, \quad n\ge1, \quad a\ (=\pi y)\ge1.
$$
 Therefore, we may replace
 $\varphi(y)$ by $\psi(y)$ and verify instead that
 $$
 \aligned
 2y^2\varphi(y)\varphi'(y)+1\ge 2y^2\psi(y)\psi'(y)+1=\\-\pi y^2 \frac{\cosh(\pi y/2)}{\sinh^3(\pi y/2)}+1=:-g(y)+1>0.
 \endaligned
 $$
 The function $g(y)>0$ with derivative
 $$
 g'(y)=
 -\frac{t(4\cosh t(t\cosh t -\sinh t )+2t)}{(\cosh^2 t -1)^2}\,\bigg|_{t=\pi y/2}<0\quad
 \text{for}\quad t\ge1
 $$
 is monotone decreasing to $0$, which finally gives for $y\ge\pi$
 $$
  2y^2\varphi(y)\varphi'(y)+1>-g(\pi)+1=-0.0064\dots+1>0
  $$
and completes the proof.
\end{proof}

The idea of the proof for  the 3D lattice sum is similar
and reduces to the 2D case  even technically.

\begin{proposition}\label{S:monotonicity-T3}
The following inequality holds for  $p>3/2$ and all $m\ge0$
\begin{equation}\label{0.pm3}
I_p(m):=m^{2p-3}
\sum_{n\in{\mathbb Z}_0^3}\frac1{(m^2+|n|^2)^p}<
\frac{\Gamma(p-3/2)\pi^{3/2}}{\Gamma(p)}.
\end{equation}
\end{proposition}

\begin{proof}
It can easily be shown by the Poisson summation formula that
$$
\lim_{m\to\infty}I_p(m)=\int_{\mathbb{R}^3}\frac {dx}
{(|x|^2+1)^p}=\frac{\Gamma(p-3/2)\pi^{3/2}}{\Gamma(p)},
$$
so that inequality \eqref{0.pm3} will be proved once we have shown that
$I_p(m)$ is monotone increasing with respect to $m\in[0,\infty)$.
The proof of monotonicity, in turn,  essentially reduces to that for the 2D
torus.

Using \eqref{funcrel}, we write
$$
\aligned
I_p(m)=
m^{2p-3}\int_0^\infty x^{p-1}e^{-m^2x}
\left(\theta_3^3(e^{-x})-1\right)\,dx=\\=
\frac1{m^{3}}\int_0^\infty x^{p-1}e^{-x}
\left(\theta_3^3(e^{-\frac{x}{\pi m^2}})-1\right)\,dx=\\=
\int_0^\infty x^{p-1}e^{-x}
\left(\frac1{m^3}\varphi^3\left(\frac x{\pi m^2}\right)-\frac1{m^3}\right)\,dx=\\=
\int_0^\infty x^{p-1}e^{-x}
\left(\frac{\pi^{3/2}}{x^{3/2}}\varphi^3
\left(\frac {\pi m^2}x\right)-\frac1{m^3}\right)\,dx.
\endaligned
$$
Therefore
$$
\frac{d}{dm}I_p(m)=\frac1{m^4}
\int_0^\infty x^{p-1}e^{-x}
\left(6y^{5/2}\varphi^2(y)\varphi'(y)
+4\right)\,dx,\quad y=y(x):=\frac{\pi m^2}x,
$$
and it suffices to show that
\begin{equation}\label{suff3}
3y^{5/2}\varphi^2(y)\varphi'(y)+2>0,\quad y\in\mathbb{R}_+.
\end{equation}

We first consider the case when $y$ is small.
Using \eqref{funcrel}, \eqref{condmon} and \eqref{fipsi} we have
for $y\le\pi$
$$
\aligned
3y^{5/2}\varphi^2(y)\varphi'(y)+2=
3\varphi(1/y)y^{2}\varphi(y)\varphi'(y)+2>\\>
-\frac32\varphi(1/y)+2>
-\frac32\psi(1/y)+2>0
\endaligned
$$
if
$$
\psi(1/y)\le\frac43\,,
$$
that is, if, see~\eqref{fipsi}
$$
y\le y_*:=\frac\pi {2\operatorname{arcoth}(4/3)}=1.6144\dots\ (<\pi),
\ \psi(1/y_*)=\frac43.
$$

On the interval $y\in (y_*,\infty)$ we have
$$
y^{5/2}\varphi^2(y)\varphi'(y)>y^{5/2}\psi^2(y)\psi'(y)=-\pi y^{5/2}
\frac{\cosh^2(\pi y/2)}{\sinh^4(\pi y/2)}:=-h(y).
$$
The function $h(y)$ with derivative
$$
h'(y)=-\frac{\pi^{1/2}\cosh t\bigl(\cosh t (4t\cosh t -5\sinh t)+4t\bigr)}
{2^{1/2}(\cosh^2 t-1)^2\sinh t},\quad t=\frac{\pi y}2
$$
is monotone decreasing for $t>5/4$  ($y>5/(2\pi)<y_*$), and for
$y\in[y^*,\infty)$
$$
-h(y)>-h(y_*)=-0.270\dots>-\frac23\,,
$$
which completes the proof of \eqref{suff3} and the proposition.
\end{proof}

\begin{remark}\label{R:T3}
{\rm
For $p=2$ direct proofs of inequalities \eqref{0.pm} and~\eqref{0.pm3} were given in
\cite{IZLap70,arxiv}, respectively,  where they were used in deriving explicit upper bounds for the dimension of the attractors
for regularized damped Euler equations in dimension $2$ and $3$, respectively.
}
\end{remark}


\begin{thebibliography}{99}



\bibitem{Araki} H.~Araki, On an inequality of Lieb and
    Thirring.
    \emph{Lett. Math. Phys.} \textbf{19} (1990), no.~2,
    167--170.



\bibitem{B-V83} A.~V. Babin and M.~I. Vishik, Attractors of
    partial     differential equations and estimates of their
    dimension.
    \emph{Uspekhi Mat. Nauk.} \textbf{38} (1983), 133--187; English
    transl. in \emph{Russian Math. Surveys} \textbf{38} (1983).

\bibitem{B-V}A.~V. Babin and M.~I. Vishik,
    \emph{Attractors of evolution equations}. Nauka,
    Moscow, \textrm 1988; \textrm {English transl.}
 North-Holland, Amsterdam, 1992.

\bibitem
{BFR80}
 J. Bardina, J. Ferziger, and  W. Reynolds,
Improved subgrid scale models for large eddy simulation.
  Proceedings of the 13th AIAA Conference on Fluid and Plasma Dynamics, (1980).

\bibitem{BDZel-Edinb} M. Bartuccelli, J. Deane and S. Zelik,
    Asymptotic expansions and extremals for the critical Sobolev
    and Gagliardo--Nirenberg inequalities on a torus. \emph{Proc.
    Royal Soc.  Edinburgh} \textbf{143A} (2013), 445--482.



\bibitem{Camassa} C. Foias, D.~D.  Holm, and  E.~S. Titi, The three
dimensional viscous Camassa--Holm equations, and their relation
to the Navier--Stokes equations and turbulence theory.
\emph{J. Dynam. Diff. Eqns} \textbf{14} (2002), 1--35.

\bibitem{Ch-I2001} V.~V. Chepyzhov and A.~A. Ilyin, A note on the
fractal dimension of attractors of dissipative dynamical
systems. \emph{Nonlinear Anal.} \textbf{44}  (2001), 811--819.


\bibitem{Ch-I} V.~V. Chepyzhov  and A.~A. Ilyin, On
the fractal dimension of invariant sets; applications to
Navier--Stokes equations. \emph{Discrete and Continuous
Dynamical Systems} \textbf{10} (2004), nos.~1\&2,   117--135.

\bibitem{C-F85} P.~Constantin and C.~Foias, Global
Lyapunov exponents, Kaplan--Yorke formulas and the dimension of
the attractors for the 2D Navier--Stokes equations.
\emph{Comm. Pure Appl. Math.} \textbf{38}  (1985), 1--27.


\bibitem{DLL} J.~Dolbeault, A.~Laptev and M.~Loss,
   Lieb--Thirring inequalities with improved constants.
 \emph{J. European Math. Soc.} \textbf{10} (2008), 1121--1126.

\bibitem{FHJN} R.~L.Frank,  D. Hundertmark, M. Jex,  Phan Th\`anh
    Nam. The Lieb--Thirring inequality revisited. \emph{J. European
    Math. Soc.} \textbf{23} (2021), 2583--2600.



\bibitem{lthbook} R.~L. Frank, A. Laptev, and T. Weidl,
\emph{Lieb--Thirring inequalities.} Cambridge University
    Press, Cambridge, (2022) in press.


\bibitem{Tr} J.~A.Hempel, G.~R. Morris, and N.~S.Trudinger, On the
    sharpness of a limiting case of the Sobolev imbedding theorem.
    \emph{Bull. Austral. Math. Soc.} \textbf{3} (1970), 369--373.

\bibitem{HLT10}
 M. Holst, E. Lunasin, and G. Tsogtgerel,
Analysis of a general family of regularized Navier-Stokes and MHD models.
  \emph{J. Nonlinear Sci.} \textbf{20}  (2010), no.~5,  523--567.

\bibitem{HLW} D. Hundertmark, A. Laptev, and T. Weidl, New
    bounds on the Lieb-Thirring constants. \textit{Invent. Math.}
    \textbf{140}(3) (2000), 693--704.


\bibitem{Il_Stokes} \textrm{A.~A. Ilyin,} On the spectrum
    of the Stokes operator.
\emph{Funktsional. Anal. i Prilozhen.}
\textbf{43} (2009), no.~4, 14--25; English transl. in
    \textit{Funct. Anal. Appl.} \textbf{43} (2009), no.~4.

\bibitem{arxiv}
 A.~A. Ilyin, A.~G. Kostianko,  S.~V. Zelik,
Sharp upper and lower bounds of the attractor dimension for   3D
damped Euler--Bardina equations. p. 31.
{http://arxiv.org/abs/math/2106.09077}, \emph{Physica D: Nonlinear
Phenomena} to appear.




 \bibitem{ILZ-JFA}\textrm{A. Ilyin, A. Laptev and S. Zelik,}
Lieb--Thirring constant on the sphere and on the torus.
\emph{J. Func. Anal.} \textbf{279}  (2020) 108784.


 \bibitem{IZLap70} A.~A. Ilyin and S.~V. Zelik,
Sharp dimension  estimates of the attractor  of
 the damped  2D Euler-Bardina equations. In
 \emph{Partial Differential Equations, Spectral Theory,
    and Mathematical Physics}, pp. 209--229, European Math. Soc. Press, Berlin, 2021.





\bibitem{Kelliher} J.~P. Kelliher, Eigenvalues of the Stokes
operator versus the Dirichlet Laplacian in the plane.
\emph{Pacific     J. Math.}  \textbf{244}  (2010), no.~1, 99--132.



\bibitem{Lad92} O.~A. Ladyzhenskaya,
First boundary     value problem for Navier--Stokes equations in domain with non
    smooth boundaries. \emph{C. R. Acad. Sc. Paris} \textbf{314}
    {\rm serie 1} (1992), 253--258.


\bibitem{Larios}
 A. Larios, B. Wingate, M. Petersen, E.~S. Titi,
The Euler-Voigt equations and a computational investigation of the
finite-time blow-up of solutions to the 3D Euler Equations
\emph{Theor. Comp. Fluid Dyn.} \textbf{3} (2018),no.~1, 23--34.

\bibitem{LiebJFA} E.~H. Lieb, An $L^p$ bound for the
Riesz and Bessel potentials of orthonormal functions.
\emph{J. Func. Anal.} \textbf{51} (1983),  159--165.

\bibitem{Lieb}
E. Lieb, On characteristic exponents in turbulence.
\emph{Comm. Math.  Phys.}
\textbf{92} (1984) 473--480.


\bibitem{Lieb--Loss} \textrm{E.\,Lieb, M.\,Loss,}
    \textit{Analysis.} Second edition. Graduate Studies in
    Mathematics, 14.
 American Mathematical Society,
 Providence, RI, 2001.

\bibitem{LT} E. Lieb and W. Thirring,
Inequalities for the moments of the
eigenvalues of the Schr\"o\-dinger Hamiltonian and their relation
to Sobolev inequalities, In \emph{Studies in Mathematical Physics.
Essays
in honor  of Valentine Bargmann}, pp. 269--303,
 Princeton University Press,
 Princeton NJ,  1976.

 \bibitem{Lopes}
M. Lopes Filho, H. Nussenzveig Lopes, E. Titi, A. Zang,
Convergence of the 2D Euler-$\alpha$ to Euler equations in the
Dirichlet case: indifference to boundary layers.
\emph{Phys. D} \textbf{292-293} (2015) 51--61.

 \bibitem{Metiv}
 G. Metivier,
Valeurs propres des op\'erateurs definis sur la restriction
de systems variationnels \`a des
 sous--espases.
\emph{J. Math. Pures Appl.} \textbf{57} (1978), 133--156.



\bibitem{Nasibov}
Sh.~M. Nasibov, On optimal
constants in some Sobolev inequalities and their application to
a nonlinear  Schr\"odinger equation. \emph{Dokl. Akad. Nauk  SSR.}
\textbf{307} (1989), 538--542; English transl. in Soviet
    Math. Dokl.  \textbf{40} (1990).

\bibitem{OT07}
 E. Olson and E. Titi,
Viscosity versus vorticity stretching: global well-posedness
  for a family of Navier-Stokes-$\alpha$-like models.
\emph{Nonlinear Anal.} \textbf{66}  (2007), no.~11, 2427--2458.

\bibitem{Ruelle}
 D. Ruelle,
 Large volume limit of the
distribution of characteristic exponents in turbulence. \emph{Comm.
Math. Phys.} \textbf{87} (1982), 287--302.



 \bibitem{traceSimon}
B. Simon, \emph{Trace ideals and their applications, \rm 2nd ed.}
Amer. Math. Soc., Providence RI, 2005.

\bibitem{Temam1}
 R. Temam,
 \emph{Navier-Stokes equations, Theory and numerical analysis}.
Amsterdam, North-Holland,  1977.


\bibitem{T85}
R.~Temam,
Attractors for     Navier--Stokes equations.
\textit{Research Notes in     Mathematics} \textbf{122} (1985), 272--292.

\bibitem{T}
R. Temam,
\emph{Infinite dimensional
   dynamical systems in mechanics and physics, \rm 2nd Edition}.
Sprin\-ger-Ver\-lag, New York,  1997.


\bibitem{Watson} G.~N. Watson,
\emph{A Treatise on the Theory of     Bessel Functions, \rm 2nd ed.}
 Cambridge University Press,
    Cambridge, 1995.

\bibitem{Wein83}
M. Weinstein, Nonlinear
Schr\"odinger equations and sharp interpolation estimates.
\emph{Comm. Math. Phys.} \textbf{87} (1983),  567--576.

\bibitem{IZ}
S.~V. Zelik, A.~A. Ilyin,
Green's     function asymptotics and sharp  interpolation inequalities.
\emph{Uspekhi Mat. Nauk} \textbf{69}(2014), no.~2, 23--76;
 English transl. in
 \textit{Russian Math. Surveys} \textbf{69} (2014), no.~2.


\bibitem{MZ} S.~V. Zelik, A.~A. Ilyin, and A.~G. Kostianko,
    Dimension     estimates for the attractor  of the regularized
    damped Euler  equations on the sphere.
\emph{ Mat. zametki}  \textbf{111} (2022), no.~1, 55-67; English tansl.
\emph{Math. Notes} \textbf{111} (2022), no.~1,  47--57.



 \end{thebibliography}
\end{document}